\documentclass{article}
\usepackage[utf8]{inputenc}
\usepackage[english]{babel}
\usepackage{amsmath}
\usepackage{wasysym}
\usepackage{amsfonts}
\usepackage{MnSymbol}
\usepackage{graphicx}
\graphicspath{ {images/} }
\usepackage[bottom]{footmisc}
\usepackage{amsthm}

\title{On winning strategies for Banach-Mazur games}
\author{
Anumat Srivastava \\
University of California, Los Angeles\\
}
\date{March 25, 2015}

\newtheorem{theorem}{Theorem}[section]
\newtheorem{corollary}[theorem]{Corollary}
\newtheorem{lemma}[theorem]{Lemma}

\theoremstyle{definition}
\newtheorem{definition}[theorem]{Definition}
 
\theoremstyle{remark}
\newtheorem*{remark}{Remark}

\begin{document}
\maketitle

\begin{abstract}
We give topological and game theoretic definitions and theorems necessary for defining a Banach-Mazur game, and apply these definitions to formalize the game.
We then state and prove two theorems which give necessary conditions for existence of winning strategies for players in a Banach-Mazur game.
\end{abstract}

\section{Introduction}
In this section we formulate and prove our main theorems. We assume preliminary knowledge of open and closed sets, closures of sets, and complete metric spaces.
\begin{definition}
A subset $T$ of a metric space $X$ is said to be \textit{dense} in $X$ if $\overline{T} = X$, that is, if $$T \bigcup \partial T= X$$
\end{definition}
\vspace{1pt}
\begin{theorem}[Baire Category Theorem]
\label{category}
Let $\{ U_n \}_{n = 1}^\infty$ be a sequence of dense open subsets of a complete metric space $X$. Then, $\bigcap_{n = 1}^\infty U_n$ is also dense in $X$.
\end{theorem}

\begin{proof}
Let $x \in X$ and let $\epsilon > 0$. It suffices to find $y \in B(x; \epsilon)$ that belongs to $\bigcap_{n = 1}^\infty U_n$. Indeed, then every open ball in $X$ meets $\bigcap_n U_n$, so that $\bigcap_n U_n$ is dense in $X$.

Since $U_1$ is dense in $X$, there exists $y_1 \in U_1$ such that $d(x, y_1) < \epsilon$. Since $U_1$ is open, there exists a $r_1 > 0$ such that $B(y_1; r_1) \subset U_1$. By shrinking $r_1$ we can arrange that $r_1 < 1$, and $\overline{B(y_1; r_1)} \subset U_1 \bigcap B(x; \epsilon)$. The same argument, with $B(y_1; r_1)$ replacing $B(x; \epsilon)$, produces $y_2 \in X$ and $0 < r_2 < 1/2$ such that $\overline{B(y_2; r_2)} \subset U_2 \bigcap B(y_1; r_1)$.
\begin{figure}
\centering
\includegraphics[scale=.45]{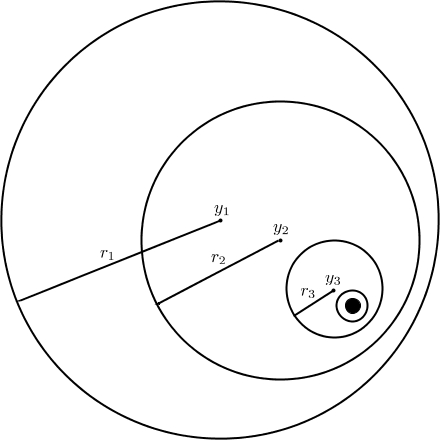}
\end{figure}

Continuing in this manner, we obtain a sequence $\{ y_n \}_{n = 1}^\infty$ in $X$ and a sequence $\{ r_n \}_{n = 1}^\infty$ or radii such that $0 < r_n < 1$ and
\begin{align}
\overline{B(y_n; r_n)} \subseteq U_n \bigcap B(y_{n-1}; r_{n-1})
\end{align}
It follows that,
\begin{align}
\overline{B(y_n; r_n)} \subset B(y_{n-1}; r_{n-1}) \subset \cdots \subset B(y_{1}; r_{1}) \subset B(x; \epsilon) 
\end{align}
The nesting property $(2)$ shows that $y_m \in B(y_n; r_n)$ if $m > n$, so that $d(y_m, y_n) < r_n \rightarrow 0 $ as $m, n \rightarrow 0$. Consequently $\{ y_m \}$ is a Cauchy sequence. Since $X$ is complete, there exists $y \in X$ such that $y_m \rightarrow y$. Since $y_m \in B(y_n; r_n)$ for $m > n$, we obtain $y \in \overline{B(y_n; r_n)}$. By $(2)$ $y \in B(x; \epsilon)$ and by $(1)$, $y \in U_n$; this is for all $n$, so that $y \in \bigcap_{n = 1}^\infty U_n$.
\end{proof}
\vspace{1pt}

\begin{definition}
A subset $Y$ of a metric space $X$ is said to be \textit{nowhere dense} if $\overline{Y}$ has no interior points, that is, if
\begin{align*}
\text{int}(Y) = \phi
\end{align*}
\end{definition}
\vspace{1pt}

\begin{corollary}
Let $\{ E_n \}_{n = 1}^\infty$ be a sequence of nowhere dense subsets of a complete metric space $X$. Then $\bigcup_{n = 1}^\infty E_n$ has empty interior.
\end{corollary}

\begin{proof}
The proof is a simple application of the Baire Category Theorem to the dense open subsets $U_n = X \backslash \overline{E_n}$
\end{proof}
\vspace{1pt}

\begin{remark}
An explanation of the nomenclature is in order. A subset of a metric space $X$ is said to be of the \textit{first category} (or \textit{meagre}) if it is the countable union of nowhere dense subsets. A subset that is not of the first category is said to be of the \textit{second category}. The complement of a meagre set is a \textit{comeagre set}.\\
Note: Second category is \textit{not} equivalent to comeagre – a set of second category may be neither meagre nor comeagre.
\end{remark}
\vspace{1pt}

\begin{definition}
A topological space $X$ is said to be a \textit{Baire space} if for every countable family $\left\{U_0, U_1, U_2, \dots \right\}$ of open and dense subsets of $X$, the intersection $\bigcap \limits_{n=0}^\infty U_n$ is dense in $X$ (equivalently if every nonempty open subset of $X$ is of second category in $X$).
\end{definition}
\vspace{1pt}
\begin{remark}
By the Baire category theorem, every complete metric space is a Baire space.
\end{remark}
\vspace{1pt}

\begin{definition}
In the mathematical study of combinatorial games, \textit{positional games} are games described by a finite set of positions in which a move consists of claiming a previously-unclaimed position.\\
Formally, a positional game is a pair $(X, \mathcal{F})$, where $X$ is a finite set of positions and $\mathcal{F}$ is a family of subsets of $X$; $X$ is called the board and the sets in $\mathcal{F}$ are called winning sets.
\end{definition}
\vspace{1pt}

\begin{definition}
In game theory, an extensive-form game has \textit{perfect information} if each player, when making any decision, is perfectly informed of all the events that have previously occurred.
\end{definition}
\vspace{1pt}

\section{Banach-Mazur Game}
We are now ready to formalize a Banach-Mazur Game.\\
\begin{definition}
A \textit{Banach-Mazur game} is a infinite positional game of perfect information played on a topological space between two players. Let $X$ be a topological space. We define the game $BM(X)$ with two players, Alice and Bob. They take turn choosing nested decreasing nonempty open subsets of $X$ as follows:
\begin{itemize}
\item Bob goes first by choosing a nonempty open subset $U_0$ of $X$. \footnote{
It is usual to assign the first move to Bob and seek a winning strategy for Alice.}
\item Alice then chooses a nonempty open subset $V_0 \subset U_0$.
\item At the $n$th play, where $n \ge 1$, Bob chooses an open set $U_n \subset V_{n-1}$ and Alice chooses an open set $V_n \subset U_n$.
\end{itemize}
Bob wins if $\bigcap \limits_{n=0}^\infty V_n \ne \phi$. Otherwise Alice wins.
\end{definition}
\section{Strategies in the Banach-Mazur Game}
We are interested in the following question:
\begin{itemize}
\item By making their moves carefully, can a player ensure that they will always win, irrespective of the moves the other player makes?
\end{itemize}
If the answer is yes, the player is said to have a winning strategy.\\

We define the strategy for a player to be a function $\sigma$ such that $U_0=\sigma(\phi)$ (the first move) and for each partial play of the game $(n \ge 1)$
$$U_0, V_0, U_1, V_1, \dots, U_{n-1}, V_{n-1}, \hfill (*)$$
$U_n=\sigma(U_0, V_0, U_1, V_1, \dots, U_{n-1}, V_{n-1})$ is a nonempty open set such that $U_n \subset V_{n-1}$.\\
We adopt the convention that a strategy for a player in a game depends only on the moves of the other player. Thus for the partial play of the Banach-Mazur game denoted by $(*)$ above, $U_n = \sigma(V_0, V_1, \dots, V_{n-1})$.\\

We shall now state two theorems which give us constraints for when Alice and Bob will have winning strategies. Evidently, who wins the game depends on the kind of topological space the game is being played on.
\begin{theorem}
Let $BM(X)$ be a Banach-Mazur game on a topological space $X$.
If $X$ is meagre, then Alice has a winning strategy for $BM(X)$.
\end{theorem}
\begin{proof}
Suppose $X$ is meagre, that is, $X = \bigcup_{n = 1}^\infty \{ x_n \}$, where each $\{ x_i \}$ is nowhere dense.\\
$U_0$ is the first set chosen by Bob. Alice can pick a subset $V_0 = U_{0} \backslash \{ x_i \}$ for some $i \in \mathbb{N}$. Then, Bob will pick some other subset $U_1$ of $V_0$ and Alice can pick $V_1 = U_1 \backslash \{ x_j \}$.\\
Continuing in this way, each point $x_n$ will be excluded by the set $V_{n}$, and $\bigcap_{n = 1}^\infty V_n = \phi$. Thus Alice wins.
\end{proof}
\vspace{1pt}

We now state a lemma which will help us establish a similar result for Bob, that is, if $X$ is a Baire space, Bob has a winning strategy for $BM(X)$.

\begin{lemma}
Let $X$ be a space. Let $O \subset X$ be a nonempty open set. Let $\tau$ be the set of all nonempty open subsets of $O$. Let $f: \tau \rightarrow \tau$ be a function such that for each $V \in \tau$, $f(V) \subset V$. Then there exists a disjoint collection $\mathcal{U}$ consisting of elements of $f(\tau)$ such that $\bigcup \mathcal{U}$ is dense in $O$.
\end{lemma}
\vspace{1pt}

\begin{theorem}
Let $BM(X)$ be a Banach-Mazur game on a topological space $X$. If $X$ is a Baire space, then Alice has no winning strategy in the game $BM(X)$.
\end{theorem}

\begin{proof}
Suppose that $X$ is a Baire space. Let $\sigma$ be a strategy for Alice. We show that $\sigma$ cannot be a winning strategy.

Let $U_0=\sigma(\phi)$ be the first move. For each open $V_0 \subset U_0$, $\sigma(V_0) \subset V_0$. By applying Lemma $3.2$ we obtain a disjoint collection $\mathcal{U}_0$ consisting of open sets of the form $\sigma(V_0)$ such that $\bigcup \mathcal{U}_0$ is dense in $U_0$.

For each $W=\sigma(V_0) \in \mathcal{U}_0$, we have $\sigma(V_0,V_1) \subset V_1$ for all open $V_1 \subset W$. So the function $\sigma(V_0, \cdot)$ is like the function $f$ in Lemma $3.2$. We can then apply Lemma $3.2$ to obtain a disjoint collection $\mathcal{U}_1(W)$ consisting of open sets of the form $\sigma(V_0, V_1)$ such that $\bigcup \mathcal{U}_1(W)$ is dense in $W$. Then let $\mathcal{U}_1=\bigcup_{W \in \mathcal{U}_0} \mathcal{U}_1(W)$. Based on how $\mathcal{U}_1(W)$ are obtained, it follows that $\bigcup \mathcal{U}_1$ is dense in $U_0$.

For each $n$, let $O_n=\bigcup \mathcal{U}_n$. Each $O_n$ is dense open in $U_0$. Since $X$ is a Baire space, every nonempty open subset of $X$ is of second category in $X$ (including $U_0$). Thus $\bigcap \limits_{n=0}^\infty O_n \ne \phi$. From this nonempty intersection, since the collection $\mathcal{U}_n$ are disjoint, we can extract a play of the game such that the open sets in this play of the game have one point in common, and thus Bob wins.\\
Thus the strategy $\sigma$ is not a winning strategy for Alice.
\end{proof}


\begin{thebibliography}{1}
\bibitem{3} "Banach–Mazur Game." Wikipedia. Wikimedia Foundation. Web. 5 June 2015.
\bibitem{1} Gamelin, Theodore W., and Robert Everist Greene. Introduction to Topology. Philadelphia: Saunders College Pub., 1983. Print.
\bibitem{5} Kenderov, P. S., and J. P. Revalski. "The Banach-Mazur Game and Generic Existence of Solutions to Optimization Problems." Proceedings of the American Mathematical Society: 911. Print.
\bibitem{2} Munkres, James R. Topology. 2nd ed. Upper Saddle River, NJ: Prentice Hall, 2000. Print.
\bibitem{4} "The Banach-Mazur Game." Dan Ma's Topology Blog. 9 June 2012. Web. 5 June 2015.
\end{thebibliography}
\end{document}